\nonstopmode
\documentclass[10pt]{amsart}
\usepackage{latexsym}
\usepackage{fancyhdr}
\usepackage{amsmath, amssymb}
\usepackage[ansinew]{inputenc}
\usepackage[all]{xy}
\usepackage{pdflscape}
\usepackage{longtable}
\usepackage{rotating}
\usepackage{verbatim}
\usepackage{hyperref}
\usepackage{subfigure}
\usepackage{mathrsfs}
\usepackage{mdwlist}

\theoremstyle{plain}
\newtheorem*{lemma*}{Lemma}
\newtheorem{lemma}{Lemma}
\newtheorem*{theorem*}{Theorem}
\newtheorem{theorem}{Theorem}
\newtheorem*{proposition*}{Proposition}
\newtheorem{proposition}{Proposition}
\newtheorem*{corollary*}{Corollary}

\newtheorem*{claim*}{Claim}

\newtheorem*{conjecture*}{Conjecture}

\newtheorem*{question*}{Question}
\theoremstyle{definition}
\newtheorem*{definition*}{Definition}

\newtheorem*{example*}{Example}

\newtheorem*{algorithm*}{Algorithm}
\newtheorem*{remark*}{Remark}
\newtheorem*{remarks*}{Remarks}
\newtheorem{remark}{Remark}

\newtheorem*{convention*}{Convention}

\numberwithin{equation}{section}

\def\al{\alpha}
\def\be{\beta}

\def\ep{\epsilon}

\def\et{\eta}

\def\la{\lambda}

\def\rh{\rho}

\def\si{\sigma}

\def\ta{\tau}

\def\vh{\varphi}

\def\om{\omega}

\def\Ga{\Gamma}

\def\La{\Lambda}

\def\Om{\Omega}

\def\C{\mathbb{C}}

\def\N{\mathbb{N}}

\def\R{\mathbb{R}}

\def\cA{\mathcal{A}}
\def\cB{\mathcal{B}}

\def\cE{\mathcal{E}}
\def\cF{\mathcal{F}}

\def\cH{\mathcal{H}}

\def\fM{\mathfrak{M}}

\def\fW{\mathfrak{W}}

\def\id{\on{id}}

\def\EM{\cE^{[M]}}
\def\EbM{\cE^{(M)}}
\def\ErM{\cE^{\{M\}}}
\def\Eom{\cE^{[\om]}}
\def\Ebom{\cE^{(\om)}}
\def\Erom{\cE^{\{\om\}}}
\def\EfM{\cE^{[\fM]}}
\def\EbfM{\cE^{(\fM)}}
\def\ErfM{\cE^{\{\fM\}}}

\def\<{\langle}
\def\>{\rangle}
\renewcommand{\o}{\circ}

\let\on=\operatorname

\newcommand{\sr}[1]%
{\ifmmode{}^\dagger\else${}^\dagger$\fi\ifvmode
\vbox to 0pt{\vss
 \hbox to 0pt{\hskip\hsize\hskip1em
 \vbox{\hsize3cm\raggedright\pretolerance10000
 \noindent #1\hfill}\hss}\vss}\else
 \vadjust{\vbox to0pt{\vss%
 \hbox to 0pt{\hskip\hsize\hskip1em%
 \vbox{\hsize3cm\raggedright\pretolerance10000%
 \noindent #1\hfill}\hss}\vss}}\fi%
}

\providecommand{\mapsfrom}{\kern.2em%
\setbox0=\hbox{$\leftarrow$\kern-.10em\rule[0.26mm]{0.1mm}{1.3mm}}\box0%
\kern.3em}

\title[Stability properties for ultradifferentiable classes]
{Equivalence of stability properties for ultradifferentiable function classes}

\author[A.~Rainer]{Armin Rainer}

\address{A.~Rainer: 
Fakult\"at f\"ur Mathematik, Universit\"at Wien, 
Oskar-Morgenstern-Platz~1, A-1090 Wien, Austria}
\email{armin.rainer@univie.ac.at}

\author[G.~Schindl]{Gerhard Schindl}

\address{G.~Schindl: Fakult\"at f\"ur Mathematik, Universit\"at Wien, 
Oskar-Morgenstern-Platz~1, A-1090 Wien, Austria}
\email{a0304518@unet.univie.ac.at}

\begin{document}

\begin{abstract}
     We characterize stability under composition, inversion, and solution of ordinary differential equations  
     for ultradifferentiable classes, and prove that all these stability properties 
     are equivalent. 
\end{abstract}

\thanks{Supported by FWF-Projects P~26735-N25 and P~23028-N13}
\date{\today}

\maketitle

\section{Introduction}

Let $\cF$  
denote some class of smooth mappings between non-empty open subsets of Euclidean spaces (of possibly different dimension).   
We say that
\begin{itemize}
  \item \emph{$\cF$ is stable under composition} if the composite of any two $\cF$-mappings $g : U \to V$ and 
  $f : V \to W$ is an $\cF$-mapping $f \o g : U \to W$.  
  \item \emph{$\cF$ is stable under solving ordinary differential equations (ODEs)} if for any $\cF$-mapping $f : \R \times \R^n \to \R^n$
  the solution of the initial value problem $x' = f(t,x)$, $x(0) = x_0 \in \R^n$ is of class $\cF$ wherever it exists.
  \item \emph{$\cF$ is stable under inversion} if for any $\cF$-mapping $f : \R^m \supseteq U \to V \subseteq \R^n$ 
  so that $f'(x_0) \in L(\R^m,\R^n)$ is invertible at $x_0 \in U$ 
  there exist neighborhoods $x_0 \in U_0 \subseteq U$ and $f(x_0) \in V_0 \subseteq V$ and an 
  $\cF$-mapping $g : V_0 \to U_0$ such that $f \o g = \on{id}_{V_0}$.   
  \item \emph{$\cF$ is inverse closed} if $1/f \in \cF(U)$ for each non-vanishing $f \in \cF(U)$. 
\end{itemize}
In this paper we shall prove that all these stability properties are equivalent for classes of ultradifferentiable 
mappings $\cF$ satisfying some mild regularity conditions. We will treat 
\begin{itemize}
  \item the classical Denjoy--Carleman classes $\EM$ determined by a weight sequence $M=(M_k)$,
  \item the classes $\Eom$ introduced by Braun, Meise, and Taylor \cite{BMT90} determined by a weight function $\om$,
  \item the classes $\EfM$ introduced in \cite{RainerSchindl12} determined by a weight matrix $\fM$.
\end{itemize}
The brackets $[~]$ stand for either $\{~\}$ in the Roumieu case or for $(~)$ in the Beurling case.
For the precise definitions we refer to Section \ref{sec:wm}. 

There are classes $\EM$ that cannot be given in terms of a weight function $\om$ 
and vice versa; see \cite{BMM07}. The classes $\EfM$ comprise all classes $\EM$ and $\Eom$ 
and hence allow for a unified approach to the classes $\EM$ and $\Eom$.
Beyond that, they provide a convenient framework to describe unions and intersections of classical Denjoy--Carleman classes.

The characterization of the aforementioned stability properties for $\EfM$ was important for treating 
$\EfM$-differomorphism groups 
in \cite{Schindl14a}.

\subsection{Stability properties of \texorpdfstring{$\EM$}{EM}}
We assume from now on that any \emph{weight sequence} $M=(M_k)$ is positive, $1= M_0 \le M_1$, and $k \mapsto k! M_k$ is log-convex 
(alias $M$ is weakly log-convex).

\begin{remark} \label{rem:wlc}
  For any weight sequence $M=(M_k)$ the sequence $(k! M_k)^{1/k}$ is increasing and $M_j M_k \le \binom{j+k}{j} M_{j+k}$ 
  for all $j,k \in \N$. 
\end{remark}

\begin{theorem} \label{thm:rM}
  If $\varliminf M_k^{1/k}>0$ and  
  $\sup (\frac{M_{k+1}}{M_k})^{1/k}<\infty$ 
  the following are equivalent:
  \begin{enumerate}
  \item $M_k^{1/k}$ is almost increasing, i.e., $\exists C>0~ \forall j \le k : M_j^{1/j} \le C M_k^{1/k}$. 
  \item $M$ has the (FdB)-property, i.e., 
  $\exists C>0 : M^\o_k \le C^k M_k$, where 
\[
M^\o_k := \max\{M_jM_{\al_1}\dots M_{\al_j}:  \al_i\in \N_{>0}, \al_1+\dots+\al_j = k\}, \quad M^\o_0:=1.
\]  
  \item $\cE^{\{M\}}$ is stable under composition.
  \item $\cE^{\{M\}}$ is stable under solving ODEs.
  \item $\cE^{\{M\}}$ is stable under inversion. 
  \item $\cE^{\{M\}}$ is inverse-closed. 
\end{enumerate}
\end{theorem}

Note that $\varliminf M_k^{1/k}>0$ iff $C^\om \subseteq \ErM$, and 
$\sup (\frac{M_{k+1}}{M_k})^{1/k}<\infty$ iff $\EM$ is stable under derivation; cf.\ \cite{RainerSchindl12}. 
If we replace the first condition by $\lim M_k^{1/k}=\infty$ which is equivalent to $C^\om \subseteq \EbM$, 
we have the corresponding Beurling type result:

\begin{theorem} \label{thm:bM}
  If $\lim M_k^{1/k}= \infty$ and  
  $\sup (\frac{M_{k+1}}{M_k})^{1/k}<\infty$ 
  the following are equivalent:
  \begin{enumerate}
  \item $M_k^{1/k}$ is almost increasing. 
  \item $M$ has the (FdB)-property.  
  \item $\EbM$ is stable under composition.
  \item $\EbM$ is stable under solving ODEs.
  \item $\EbM$ is stable under inversion.
  \item $\EbM$ is inverse-closed. 
\end{enumerate}
\end{theorem}

Most implications of Theorems \ref{thm:rM} and \ref{thm:bM} are basically known, but scattered in the literature. 

The equivalence of (1) and (6) is due to Rudin \cite{Rudin62} in the Roumieu case and to Bruna \cite{Bruna80/81} in the Beurling case; note that Rudin 
only considered non-quasianalytic classes and H\"ormander dealt with the quasianalytic case (cf.\ \cite[p.~799]{Rudin62}). See also \cite{Siddiqi90}. 

That (1) implies stability under inversion is due to Komatsu \cite{Komatsu79}; different proofs in the Banach space setting were given by 
Yamanaka \cite{Yamanaka89} and Koike \cite{Koike96}.  
The sufficiency of (1) for stability under solving ODEs was obtained by Komatsu \cite{Komatsu80} and in Banach spaces by Yamanaka \cite{Yamanaka91}. 

That the class $\ErM$ is stable under composition, provided that 
$M=(M_k)$ is log-convex (which implies (1)), is due to Roumieu \cite{Roumieu62/63}; other references are e.g.\ \cite{Komatsu73} and \cite{BM04}.
In \cite{RainerSchindl12} we proved the equivalence of (1), (2), and (3) (in both the Beurling and the Roumieu case). 

It is worth mentioning that Dynkin \cite{Dynkin80} gave a characterization of the Roumieu classes $\ErM$ in terms of \emph{almost holomorphic 
extensions}, provided that $M=(M_k)$ is log-convex, which implies the stability properties (3), (4), (5), and (6) in a straightforward manner.   

That all the properties (1) -- (6) are equivalent was, to our knowledge, not observed before.

\subsection{Stability properties of \texorpdfstring{$\Eom$}{Eomega}}
The respective result in the weight function case, that is Theorems \ref{thm:rom} and \ref{thm:bom} below, 
was not known before, apart from a characterizaton for stability under 
composition obtained in \cite{FernandezGalbis06} and in \cite{RainerSchindl12}.

We henceforth assume that any \emph{weight function} $\om$
is a continuous increasing function $\om: [0,\infty) \to [0,\infty)$ with $\om|_{[0,1]}=0$, 
$\lim_{t \to \infty} \om(t) = \infty$, and so that: 
\begin{align} 
  \tag{$\om_1$} \label{om_1} &\om(2t)=O(\om(t)) \text{ as } t\to \infty. \\
  \tag{$\om_2$} \label{om_2} &\log(t)=o(\omega(t)) \text{ as } t\to \infty. \\
  \tag{$\om_3$} \label{om_3} &\vh : t \mapsto \om(e^t) \text{ is convex on } [0,\infty).
\end{align}
Note that $C^\om \subseteq \Erom$ iff $\om(t)=O(t)$, and $C^\om \subseteq \Ebom$ iff $\om(t)=o(t)$, as $t\to \infty$. 

\begin{theorem} \label{thm:rom}
  If $\om$ satisfies $\om(t)=O(t)$ as $t\to \infty$ then the following are equivalent:
  \begin{enumerate}
  \item $\om$ satisfies $\exists C>0 ~\exists t_0>0 ~\forall \la\ge 1 ~\forall t\ge t_0 
  : \om(\la t) \le C \la \om(t)$.
  \item There exists a sub-additive weight function $\tilde \om$ so that $\Eom = \cE^{[\tilde \om]}$.
  \item $\cE^{\{\om\}}$ is stable under composition.
  \item $\cE^{\{\om\}}$ is stable under solving ODEs.
  \item $\cE^{\{\om\}}$ is stable under inversion. 
  \item $\cE^{\{\om\}}$ is inverse-closed.
\end{enumerate}
\end{theorem} 

\begin{theorem} \label{thm:bom}
  If $\om$ satisfies $\om(t)=o(t)$ as $t\to \infty$ then the following are equivalent:
  \begin{enumerate}
  \item $\om$ satisfies $\exists C>0 ~\exists t_0>0 ~\forall \la\ge 1 ~\forall t\ge t_0 
  : \om(\la t) \le C \la \om(t)$.
  \item There exists a sub-additive weight function $\tilde \om$ so that $\Eom = \cE^{[\tilde \om]}$.
  \item $\Ebom$ is stable under composition.
  \item $\Ebom$ is stable under solving ODEs.
  \item $\Ebom$ is stable under inversion. 
  \item $\Ebom$ is inverse-closed.
\end{enumerate}
\end{theorem}

\subsection{Stability properties of \texorpdfstring{$\EfM$}{EfM}}
Theorems \ref{thm:rM}, \ref{thm:bM}, \ref{thm:rom}, and \ref{thm:bom} are corollaries of the corresponding result for 
the classes $\EfM$ defined in terms of weight matrices, namely Theorems \ref{thm:rfM} and \ref{thm:bfM} below.

A \emph{weight matrix} $\fM = \{M^\la \in \R_{>0}^{\N} : \la \in \La\}$ is a family of 
weight sequences $M^\la=(M^\la_k)$ indexed by an ordered subset $\La$ of $\R$ so that  
$\lim_{k} (k! M^\la_k)^\frac{1}{k}=\infty$ for each $\la$, 
and $M^\la\le M^\mu$ if $\la \le \mu$.

We shall again assume that the classes $\EfM$ contain the class of real analytic functions and are stable under 
derivation. Specifically we need the conditions
\begin{align}
  \tag{$\fM_{\cH}$} \label{fM_H} 
  &\forall \la \in \La : \varliminf (M^\la_k)^{\frac{1}{k}}>0 \\
  \tag{$\fM_{(C^\om)}$} \label{fM_(Com)} 
  &\forall \la \in \La : \lim (M^\la_k)^{\frac{1}{k}}=\infty  \\
  \tag{$\fM_{\{\on{dc}\}}$} \label{fM_{dc}} 
  &\forall \la \in \La ~\exists \mu \in \La ~\exists C>0 
  ~\forall k \in \N: M^\la_{k+1}\le C^{k} M^{\mu}_k \\
  \tag{$\fM_{(\on{dc})}$} \label{fM_(dc)} 
  &\forall \la \in \La ~\exists \mu \in \La ~\exists C>0 
  ~\forall k \in \N: M^\mu_{k+1}\le C^{k} M^{\la}_k
\end{align}
We have 
\[
\xymatrix{
  \eqref{fM_(Com)} \ar@{<=>}[r] \ar@{=>}[d] & C^\om(U) \subseteq \EbfM(U) \ar@{=>}[d]\\
  \eqref{fM_H} \ar@{<=>}[r] \ar@{=>}[d] & \cH(\C^n) \subseteq \cE^{(\fM)}(U) \ar@{=>}[d]\\
  (\fM_{\{C^\om\}}) :\Leftrightarrow \exists \la \in \La : \varliminf (M^\la_k)^{\frac{1}{k}}>0 \ar@{<=>}[r] & C^\om(U) \subseteq \ErfM(U)
}
\]
and 
$\cE^{[\fM]}(U)$ is derivation closed iff \thetag{$\fM_{[\on{dc}]}$}; see \cite{RainerSchindl12}. 
The conditions on the weight matrix $\fM$ that characterize the stability properties of $\EfM$ are 
natural generalizations of the condition of being almost increasing and of the (FdB)-property; 
clearly, the Roumieu and the Beurling version fall apart, see Remark~\ref{rem:rai} below:   
\begin{align}
  \tag{$\fM_{\{\on{rai}\}}$} \label{fM_{rai}} 
  &\forall \la \in \La ~\exists \mu \in \La ~\exists C>0 ~\forall j \le k: 
  (M^{\la}_j)^{\frac{1}{j}} \le C (M^{\mu}_k)^{\frac{1}{k}} \\
  \tag{$\fM_{(\on{rai})}$} \label{fM_(rai)}
  &\forall \la \in \La ~\exists \mu \in \La ~\exists C>0 ~\forall j \le k: 
  (M^{\mu}_j)^{\frac{1}{j}} \le C (M^{\la}_k)^{\frac{1}{k}} \\
  \tag{$\fM_{\{\on{FdB}\}}$} \label{fM_{FdB}} 
  &\forall \la \in \La ~\exists \mu \in \La ~\exists C>0 ~\forall k: (M^\la)^\o_k \le C^k M^\mu_k \\
  \tag{$\fM_{(\on{FdB})}$} \label{fM_(FdB)}
  &\forall \la \in \La ~\exists \mu \in \La ~\exists C>0 ~\forall k: (M^\mu)^\o_k \le C^k M^\la_k
\end{align}

\begin{theorem} \label{thm:rfM}
For a weight matrix $\fM$ satisfying \eqref{fM_H} and \eqref{fM_{dc}}
the following are equivalent:
\begin{enumerate}
  \item $\fM$ satisfies \eqref{fM_{rai}}.
  \item $\fM$ satisfies \eqref{fM_{FdB}}.
  \item $\cE^{\{\fM\}}$ is stable under composition.
  \item $\cE^{\{\fM\}}$ is stable under solving ODEs.
  \item $\cE^{\{\fM\}}$ is stable under inversion. 
  \item $\cE^{\{\fM\}}$ is inverse-closed.
\end{enumerate}
\end{theorem}

\begin{theorem} \label{thm:bfM}
For a weight matrix $\fM$ satisfying \eqref{fM_(Com)} and \eqref{fM_(dc)}
the following are equivalent:
\begin{enumerate}
  \item $\fM$ satisfies \eqref{fM_(rai)}.
  \item $\fM$ satisfies \eqref{fM_(FdB)}.
  \item $\EbfM$ is stable under composition.
  \item $\EbfM$ is stable under solving ODEs.
  \item $\EbfM$ is stable under inversion. 
  \item $\EbfM$ is inverse-closed.
\end{enumerate}
\end{theorem}

\begin{remark} \label{rem:rai}
  The weight matrix $\fM$ that consists of just two non-equivalent sequences 
  $M^1 \le M^2$ satisfying 
  \begin{itemize}
    \item $(M^i_k)^{1/k} \to \infty$ and $\sup_k (\frac{M^i_{k+1}}{M^i_k})^{1/k}< \infty$, $i=1,2$, 
    \item $(M^1_k)^{1/k}$ almost increasing,
    \item $(M^2_k)^{1/k}$ not almost increasing,
  \end{itemize}
  satisfies \eqref{fM_(rai)} but not \eqref{fM_{rai}}. 
  Whereas, if $(M^2_k)^{1/k}$ is almost increasing and $(M^1_k)^{1/k}$ is not, 
  $\fM$ satisfies \eqref{fM_{rai}} but not \eqref{fM_(rai)}.
  We construct such sequences in Appendix~\ref{appendix}.
\end{remark}

We shall prove Theorems \ref{thm:rfM} and \ref{thm:bfM} in Sections \ref{sec:proofR} and \ref{sec:proofB}, 
respectively. In Section \ref{sec:proofMom} we show that Theorems \ref{thm:rM}, \ref{thm:bM}, \ref{thm:rom}, and \ref{thm:bom} 
are corollaries of Theorems \ref{thm:rfM} and \ref{thm:bfM}.

\subsection*{Notation}

The notation $\cE^{[*]}$ for $* \in \{M,\om,\fM\}$ stands for either $\cE^{(*)}$ or $\cE^{\{*\}}$ with the following restriction: 
Statements that involve more than one $\cE^{[*]}$ symbol must not be interpreted by mixing $\cE^{(*)}$ and $\cE^{\{*\}}$.

\section{Ultradifferentiable function classes} \label{sec:wm}

\subsection{Ultradifferentiable functions defined by weight sequences}

Let $M=(M_k)$ be a weight sequence.
For non-empty open $U \subseteq \R^n$, define
\begin{align*}
  \cE^{(M)}(U,\R^m) &:= \Big\{f \in C^\infty(U,\R^m) : \forall K \subseteq U \text{ compact} ~\forall \rh>0 : \|f\|^M_{K,\rh} < \infty \Big\} \\
  \cE^{\{M\}}(U,\R^m) &:= \Big\{f \in C^\infty(U,\R^m) : \forall K \subseteq U \text{ compact} ~\exists \rh>0 : \|f\|^M_{K,\rh} < \infty \Big\} \\
  &\|f\|^M_{K,\rh} := \sup\Big\{\frac{\|f^{(k)}(x)\|_{L^k(\R^n,\R^m)}}{\rh^k k!        M_k}:x\in K,k\in\N\Big\} 
\end{align*}
and endow these spaces with their natural topologies:
\begin{gather*}
  \cE^{(M)}(U,\R^m) = \varprojlim_{K \subseteq U} \varprojlim_{\ell \in \N_{>0}} \cE^M_{\frac{1}{\ell}}(K,\R^m), \quad 
  \cE^{\{M\}}(U,\R^m) = \varprojlim_{K \subseteq U} \varinjlim_{\ell \in \N} \cE^M_{\ell}(K,\R^m) 
  \\
  \text{where }\quad  \cE^M_\rh(K,\R^m) := \{f \in C^\infty(K,\R^m) : \|f\|^M_{K,\rh} < \infty \}
\end{gather*}
We will need the following inclusion relations (cf.\ \cite{RainerSchindl12}):
\begin{align*}
 \cE^{[M]} \subseteq \cE^{[N]} \quad &\Leftrightarrow \quad 
 M \preceq N \quad :\Leftrightarrow \quad  \exists C,\rh > 0 ~\forall k : M_k \le C \rh^k N_k   \\ 
 \cE^{\{M\}} \subseteq \cE^{(N)} \quad &\Leftrightarrow \quad 
 M \lhd N \quad :\Leftrightarrow \quad \forall \rh>0 ~\exists C>0 ~\forall k : M_k \le C \rh^k N_k    
\end{align*}
In particular,  
$C^\om(U) \subseteq \cE^{\{M\}}(U) \Leftrightarrow \cH(\C^n) \subseteq \cE^{(M)}(U) 
  \Leftrightarrow \varliminf M_k^{\frac{1}{k}}>0$ and $C^\om(U) \subseteq \cE^{(M)}(U) \Leftrightarrow \lim M_k^{\frac{1}{k}} = \infty$.

\subsection{Ultradifferentiable functions defined by weight functions} \label{ssec:wf}

Let $\om$ be a weight function (hence satisfying \eqref{om_1}, \eqref{om_2}, and \eqref{om_3}).
The \emph{Young conjugate} of $\vh(t) = \om(e^t)$, given by 
\[
\vh^*(t):=\sup \{st-\vh(s) : s \ge 0\}, \quad t \ge 0,
\] 
is convex, increasing, and satisfies $\vh^*(0)=0$, $\vh^{**}=\vh$,  
and $\lim_{t\to \infty} t/\vh^*(t)=0$. 
Moreover,
the functions $t \mapsto \vh(t)/t$ and $t \mapsto \vh^*(t)/t$ are increasing; see e.g.\ \cite{BMT90}. 
For non-empty open $U \subseteq \R^n$ define
\begin{align*}
  \cE^{(\om)}(U,\R^m) 
  &:= \Big\{f \in C^\infty(U,\R^m) : 
  \forall K \subseteq U \text{ compact} ~\forall \rh > 0 : \|f\|^\om_{K,\rh} < \infty \Big\} \\
  \cE^{\{\om\}}(U,\R^m) 
  &:= \Big\{f \in C^\infty(U,\R^m) : \forall K \subseteq U \text{ compact} ~\exists \rh>0 : \|f\|^\om_{K,\rh} < \infty \Big\} \\
  &\|f\|^\om_{K,\rh} := \sup\Big\{\|f^{(k)}(x)\|_{L^k(\R^n,\R^m)} \exp(- \tfrac{1}{\rh} \vh^*(\rh k)) : x\in K,k\in\N\Big\} 
\end{align*}
and endow these spaces with their natural topologies: 
\begin{gather*}
  \Ebom(U,\R^m) = \varprojlim_{K \subseteq U} \varprojlim_{\ell \in \N_{>0}} \cE^\om_{\frac{1}{\ell}}(K,\R^m), 
  \quad
  \Erom(U,\R^m) = \varprojlim_{K \subseteq U} \varinjlim_{\ell \in \N} \cE^\om_{\ell}(K,\R^m) 
  \\ \text{where }\quad 
  \cE^\om_\rh(K,\R^m) := \{f \in C^\infty(K,\R^m) : \|f\|^\om_{K,\rh} < \infty \}
\end{gather*}
We have 
$C^\om(U) \subseteq \cE^{\{\om\}}(U)$ iff $\om(t)=O(t)$ as $t \to \infty$, and 
$C^\om(U) \subseteq \cE^{(\om)}(U)$ iff $\om(t)=o(t)$ as $t \to \infty$; see e.g.\ \cite{RainerSchindl12}.

\subsection{Ultradifferentiable functions defined by weight matrices} \label{ssec:fM}

Let $\fM$ be a weight matrix, let $U \subseteq \R^n$ be non-empty and open, and let $K \subseteq U$ be compact. 
We define
\begin{gather*}
\cE^{(\fM)}(K,\R^m) := \bigcap_{\la \in \La} \cE^{(M^{\la})}(K,\R^m), \quad 
 \cE^{\{\fM\}}(K,\R^m) := \bigcup_{\la \in \La} \cE^{\{M^{\la}\}}(K,\R^m), \\
\cE^{(\fM)}(U,\R^m) := \bigcap_{\la \in \La} \cE^{(M^{\la})}(U,\R^m), \quad 
 \cE^{\{\fM\}}(U,\R^m) := \bigcap_{K \subseteq U}\bigcup_{\la \in \La} \cE^{\{M^{\la}\}}(K,\R^m),
\end{gather*}
and endow these spaces with their natural topologies:
\begin{align*}
  \cE^{(\fM)}(U,\R^m) := \varprojlim_{\la\in \La} \cE^{(M^{\la})}(U,\R^m), \quad  
  \cE^{\{\fM\}}(U,\R^m) := \varprojlim_{K \subseteq U} \varinjlim_{\la \in \La} \cE^{\{M^{\la}\}}(K,\R^m).  
\end{align*}
It is no loss of generality to assume that the limits are countable.

We have $\cH(\C^n) \subseteq \cE^{(\fM)}(U)$ iff \eqref{fM_H}, 
$C^\om(U) \subseteq \cE^{[\fM]}(U)$ iff \thetag{$\fM_{[C^\om]}$}, and 
$\cE^{[\fM]}(U)$ is derivation closed iff \thetag{$\fM_{[\on{dc}]}$}; see \cite{RainerSchindl12}.

\section{Proof of Theorem \ref{thm:rfM}: the Roumieu case} \label{sec:proofR}

The proof of the equivalence of the items (1)--(6) of Theorem \ref{thm:rfM} has the following structure: 
\begin{equation} \label{eq:diag}
\begin{split}
   \xymatrix@R=.2cm{
    && (4) \ar@{=>}[dl] &&&\\
    & (6') \ar@{=>}[rr] && (1) \ar@{=>}[ul] \ar@{=>}[dl] \ar@{<=>}[r] & (2) \ar@{<=>}[r] & (3) \\
    (6) \ar@{=>}[ur] && (5) \ar@{=>}[ll] & & & 
  }
\end{split} 
\end{equation}
where: 
\begin{enumerate}
    \item[$(6')$] If $f \in \cE^{\{\fM\}}(\R)$ and $f(0) \ne 0$ then $1/f$ is $\ErfM$ on its domain of definition.  
\end{enumerate}
We successively prove:
\begin{itemize}
   \item the equivalences $(1) \Leftrightarrow (2) \Leftrightarrow (3)$ 
   \item the cycle $(1) \Rightarrow (4) \Rightarrow (6') \Rightarrow (1)$
   \item the cycle $(1) \Rightarrow (5) \Rightarrow (6) \Rightarrow (6') \Rightarrow (1)$
 \end{itemize}

\subsection{The equivalences \texorpdfstring{$(1) \Leftrightarrow (2) \Leftrightarrow (3)$}{(1)<=>(2)<=>(3)}}

The following 
lemma implies the equi\-valence of (1) and (2). 
The equivalence of (2) and (3)
was shown in \cite[4.9]{RainerSchindl12}.

\begin{lemma} \label{lem:raiFdB}
For a weight matrix $\fM$ we have the following implications:
\begin{enumerate}
    \item \thetag{\hyperref[fM_(rai)]{$\fM_{[\on{rai}]}$}} and \thetag{\hyperref[fM_(dc)]{$\fM_{[\on{dc}]}$}} imply 
    \thetag{\hyperref[fM_(FdB)]{$\fM_{[\on{FdB}]}$}}.
    \item \thetag{\hyperref[fM_(FdB)]{$\fM_{[\on{FdB}]}$}} and \eqref{fM_H} imply 
    \thetag{\hyperref[fM_(rai)]{$\fM_{[\on{rai}]}$}}
\end{enumerate}  
\end{lemma}

\begin{proof}
  \thetag{1} was shown in \cite[4.9, 4.11]{RainerSchindl12}.
  To see \thetag{2} observe that $(M^\mu)^\o \preceq M^\la$ implies 
  $(M^\mu_j)^{\frac{1}{jk}} (M^\mu_k)^{\frac{1}{k}} \le C (M^\la_{jk})^{\frac{1}{jk}}$ for all $j,k$ and some constant $C$.
  By \eqref{fM_H} we may conclude that 
  $(M^\mu_k)^{\frac{1}{k}} \le \tilde C (M^\la_{jk})^{\frac{1}{jk}}$ for all $j,k$ and some $\tilde C$. 
  For $k \le \ell$ choose $j \in \N$ such that $jk \le \ell < (j+1)k$, then by Remark \ref{rem:wlc} and 
  since $n! \le n^n \le e^n n!$,
  \[
    \ell (M^\la_\ell)^{\frac{1}{\ell}} \!\ge\!
    (\ell! M^\la_\ell)^{\frac{1}{\ell}} \!\ge\! ((jk)! M^\la_{jk})^{\frac{1}{jk}} \ge \frac{\tilde C jk}{e}  (M^\mu_{k})^{\frac{1}{k}} 
    \ge \frac{\tilde C (j+1)k}{2e}  (M^\mu_{k})^{\frac{1}{k}} >  \frac{\tilde C\ell}{2e}  (M^\mu_{k})^{\frac{1}{k}}
  \]
  which implies the desired property. 
\end{proof}

\subsection{The cycle \texorpdfstring{$(1) \Rightarrow (4) \Rightarrow (6') \Rightarrow (1)$}{(1)=>(4)=>(6')=>(1)}} \label{ssec:cycle1}

The implication (1) $\Rightarrow$ (4)
follows from the following proposition.
We state and prove this result on ultradifferentiable 
solutions of ODEs (as well as the ultradifferentiable inverse mapping theorem below) 
for mappings between arbitrary Banach spaces, since we used such results in \cite{KMRu} and will 
need them in forthcoming work; cf.\ \cite{Schindl14a}.

\begin{remark} \label{rem:shift}
  Observe the index shift in the estimates of \eqref{eq:f} and \eqref{eq:x}. 
  In order to deduce (1) $\Rightarrow$ (4) from Proposition \ref{prop:ODEr} we 
  use \eqref{fM_{dc}} for the index shift in \eqref{eq:f} and Remark \ref{rem:wlc} 
  for the one in \eqref{eq:x}.
\end{remark}

Henceforth we use the convention $(-1)! M_{-1} := 1$ for any weight sequence $M$.

\begin{proposition}\label{prop:ODEr}
Let $\fM$ be a weight matrix satisfying  
\eqref{fM_{rai}}.
Let $X$ be a Banach space and let $f : W \to X$ be a $C^\infty$-mapping defined in an open 
subset $W \subseteq X \times \R$ and satisfying
\begin{align} \label{eq:f}
\begin{split}
  \exists \la \in \La ~\exists C,\rh \ge1 &~\forall (k,\ell) \in \N^2 ~\forall  (x,t) \in W : 
  \\
  &\|f^{(k,\ell)}(x,t)\|_{L^{k,\ell}(X,\R;X)} \le C \rh^{k+\ell} (k+\ell-1)! M^\la_{k+\ell-1}.
\end{split}
\end{align}
Then the solution $x : I \to X$ of the initial value problem
\begin{equation} \label{eq:IVP}
  x'(t) = f(x(t),t), \quad x(0) = x_0,  
\end{equation}
satisfies
\begin{align} \label{eq:x}
  \exists \mu \in \La ~\exists D,\si \ge1 ~\forall k \in \N ~\forall  t \in I : 
  \|x^{(k)}(t)\| \le D \si^{k} (k-1)! M^\mu_{k-1}.
\end{align} 
\end{proposition}

\begin{remark} \label{rem:1}
A more general statement involving parameters $u$ in a further Banach space $Z$ is true:
the solution of the initial value problem
\begin{equation*} 
  x'(t) = f(x(t),t,u), \quad x(0) = x_0,  
\end{equation*}
satisfies an estimate of the kind \eqref{eq:x} in $t$, $u$, and $x_0$,
given that $f$ satisfies an estimate of the kind \eqref{eq:f} in $x$, $t$, and $u$. 
For simplicity we prove only the result stated in the proposition; 
the general result is obtained by making obvious modifications in the proof of \cite{Yamanaka91} the main ideas of 
which we follow here. Different arguments were given in \cite{Komatsu80} and \cite{Dynkin80}.  
\end{remark}

In the proof of the proposition we will use Fa\`a di Bruno's formula for Fr\'echet derivatives of mappings between Banach spaces. 
So let us recall this formula. 
For $k\ge 1$, 
\begin{align} \label{eq:FaaF}
\frac{(f\o g)^{(k)}(x)}{k!} 
= \on{sym}\Big( \sum_{j\ge 1} \!\sum_{\substack{\al\in \N_{>0}^j\\ \sum_{i=1}^j \al_i =k}}\!
\frac{f^{(j)}(g(x))}{j!} \o \Big(\frac{g^{(\al_1)}(x)}{\al_1!}\times\cdots\times\frac{g^{(\al_j)}(x)}{\al_j!}\Big)\Big),
\end{align}
where $\on{sym}$ denotes symmetrization of multilinear mappings.

\begin{proof}
  We may reduce the initial value problem \eqref{eq:IVP} to the problem
  \begin{equation} \label{eq:IVP2}
    y'= g(y), \quad y(0) =y_0,    
  \end{equation} 
  by setting $y =(x,t)$, $y_0=(x_0,0)$, and $g(y) = (f(y),1)$. 
  So we assume that $Y$ is a Banach space, 
  $U$ is a neighborhood of $0$ in $Y$, 
  and $g \in C^\infty(U,Y)$ satisfies
  \begin{align} \label{eq:g2}
     \exists \la \in \La ~\exists C,\rh \ge1 ~\forall k \in \N ~\forall  y \in U : \|g^{(k)}(y)\|_{L^k(Y,Y)} \le C \rh^k (k-1)! M^\la_{k-1}.
  \end{align} 
  Without loss of generality we assume $M^\la_1 \ge 2$.
  By the classical existence and uniqueness result, there exists a unique $C^\infty$ 
  solution $y =y(t)$ of \eqref{eq:IVP2} for $t$ 
  in a neighborhood $I$ of $0$. We assume that $\sup_{t \in I} \|y(t)\|<\infty$.

  By \eqref{fM_{rai}} there exists $\mu \in \La$ and $H\ge 1$ such that for $2 \le j \le k$,
  \begin{align*}
    \Big(\frac{M^\la_{j-1}}{j}\Big)^{\frac{1}{j-1}} \le H \Big(\frac{M^{\mu}_{k-1}}{k}\Big)^{\frac{1}{k-1}} =: p^\mu_k. 
  \end{align*} 
  Then, since $1 \le M^\la_1/2 \le p^\mu_k$, 
  \begin{align} \label{eq:p}
    \frac{M^\la_{j-1}}{j} \le (p^\mu_k)^j, \quad \text{ for } 2 \le j \le k.
  \end{align}
  Let us choose constants $A$ and $\et$ such that
  \begin{equation} \label{eq:Aet}
    A \ge \max\{\sup_{t \in I} \|y(t)\|,C\}  \quad \text{ and } \quad \et \ge \rh,    
  \end{equation}
  where $C$ and $\rh$ are the constants from \eqref{eq:g2}.
  We define
  \begin{align*}
    G^\mu_k(s) &:= \frac{A}{1-\et p^\mu_k  s},
  \end{align*}
  for small $s \in \R$,
  and consider the initial value problem
  \begin{equation} \label{eq:IVP3}
    Y'(t) = G^\mu_k(Y(t)-A), \quad Y(0) = A. 
  \end{equation}
  We claim that the solution 
  \begin{align*}
    Y^\mu_k(t) =A+ \frac{1-\sqrt{1-2A \et p^\mu_k t}}{\et p^\mu_k}
  \end{align*}
   of \eqref{eq:IVP3} satisfies
  \begin{equation} \label{eq:ind}
    \sup_{t \in I} \|y^{(j)}(t)\| \le (Y^\mu_k)^{(j)}(0) \quad \text{ for } j \le k.  
  \end{equation}
  This implies the statement of the proposition, since, for $j \ge 1$,
  \begin{align*}
    (Y^\mu_k)^{(j)}(0) &= (2A)^j (\et p^\mu_k)^{j-1} (2\sqrt{\pi})^{-1} \Ga(j-\tfrac1 2) \\
    & \le (2A)^j (\et H)^{j-1} (j-1)! M^\mu_{j-1} \quad \text{ if } k=j. 
  \end{align*}

  Let us prove \eqref{eq:ind}.
  By the choice of the constant $A$, \eqref{eq:ind} is satisfied for $j=0$.
  Suppose that \eqref{eq:ind} holds for all $j \le \ell< k$. 
  By \eqref{eq:g2}, \eqref{eq:p}, and \eqref{eq:Aet}, we have  
  \begin{align*}
    \sup_{y \in U} \|g^{(j)}(y)\|_{L^j(Y,Y)} 
    &\le C \rh^j (j-1)! M^\la_{j-1}  
    \le C \rh^j j! (p^\mu_k)^j 
    \le A \et^j j! (p^\mu_k)^j = (G^{\mu}_k)^{(j)}(0).
  \end{align*}
  So, by applying Fa\`a di Bruno's formula \eqref{eq:FaaF} twice, we may conclude that, for $j \le \ell < k$,
  \begin{align*} 
    \sup_{t \in I} \|(g \o y)^{(j)}(t)\| &\le j! \sum_{h\ge 1} \frac{(G^{\mu}_k)^{(h)}(0)}{h!} 
    \sum_{\substack{\al_1+\cdots+\al_h=j\\ \al_i>0}} 
    \prod_{i=1}^h  \frac{(Y^\mu_k)^{(\al_i)}(0)}{\al_i!} \\
    &= (G^\mu_k \o (Y^\mu_k-A))^{(j)}(0).
  \end{align*}
  But as $y$ and $Y^\mu_k$ are solutions of \eqref{eq:IVP2} and \eqref{eq:IVP3}, respectively, 
  it follows that \eqref{eq:ind} holds for $j \le \ell+1$.  
  By induction, \eqref{eq:ind} follows.  
\end{proof}

Let us check that (4) implies $(6')$. 
Let $f \in \cE^{\{\fM\}}(\R)$ satisfy $f(0) \ne 0$ and consider $g := 1/f$. 
Then $g$ solves the initial value problem
  \[
    x' = - f'(t) x^2, \quad x(0) = g(0).
  \]
By \eqref{fM_{dc}}, the mapping $(x,t) \mapsto f'(t)x^2$ is $\ErfM$ and so $g$ is $\cE^{\{\fM\}}$, by (4).

The implication $(6')$ $\Rightarrow$ (1) was shown in the proof of 
\cite[4.9(2)$\Rightarrow$(3)]{RainerSchindl12}, by following the argument of 
\cite[Thm~3]{Siddiqi90}.

The next lemma, a variation of \cite[Thm~1]{Koike96}, 
is a preparation for Proposition~\ref{prop:inverse} below. It gives, in particular, 
a direct proof of the implication (1) $\Rightarrow$ $(6')$.

\begin{lemma} \label{lem:inverseclosed} 
  Let $\fM$ be a weight matrix satisfying \eqref{fM_{rai}}. 
  Let $E, F, G$ be Banach spaces, $U \subseteq E$ be open, and let $T \in C^\infty (U,L(F,G))$ satisfy
  \begin{align} \label{eq:1/f}
  \begin{split}
  \exists \la \in \La ~\exists C,\rh \ge1 &~\forall k \in \N ~\forall  x\in U : 
  \|T^{(k)}(x)\|_{L^k(E,L(F,G))} \le C \rh^{k} k! M^\la_{k}.
  \end{split}
  \end{align}
  If $T(x_0) \in L(F,G)$ is invertible,  
  then there is a neighborhood $x_0 \in U_0 \subseteq U$ such that $U_0 \ni x \mapsto S(x) := T(x)^{-1}$  
  satisfies 
  \begin{align} \label{eq:1/g}
  \begin{split}
  \exists \mu \in \La ~\exists D,\si \ge1 &~\forall k \in \N ~\forall  x\in U_0 : 
  \|S^{(k)}(x)\|_{L^k(E,L(G,F))} \le D \si^{k} k! M^\mu_{k}.
  \end{split}
  \end{align}
\end{lemma}

\begin{proof}
  There is an open neighborhood $U_0$ of $x_0$ so that for $x \in U_0$ we have $\|S(x)\|_{L(G,F)} \le A$ for some constant
  $A>0$ 
  and $S(x)$ is given by the Neumann series   
  \[
    S(x) = T(x_0)^{-1} \sum_{j=0}^\infty \big((T(x_0)-T(x))T(x_0)^{-1}\big)^j.
  \]
  For $y$ near $x$ we may consider  
  \[
    S(y) = S(x) \sum_{j=0}^\infty \big((T(x)-T(y)) S(x)\big)^j = S(x) \big(\id - (T(x)-T(y)) S(x)\big)^{-1}
  \]
  and use Fa\`a di Bruno's formula \eqref{eq:FaaF} and \eqref{eq:1/f} to obtain, for $y=x$,
  \begin{align*}
    \frac{\|S^{(k)}(x)\|_{L^k(E,L(G,F))}}{k!} 
    &\le A \sum_{j\ge 1} \sum_{\substack{\al_1+\cdots+\al_j=k\\ \al_i>0}} 
    (AC)^j \rh^k \prod_{i=1}^j M^\la_{\al_i} 
  \end{align*}
  which implies \eqref{eq:1/g}, since $M^\la_{\al_1} \cdots M^\la_{\al_j} \le H^{k} M^\mu_{k}$, by \eqref{fM_{rai}}. 
\end{proof}

\subsection{The cycle \texorpdfstring{$(1) \Rightarrow (5) \Rightarrow (6) \Rightarrow (6') \Rightarrow (1)$}{(1)=>(5)=>(6)=>(6')=>(1)}} 
\label{ssec:cycle2}

The implications $(5) \Rightarrow (6) \Rightarrow (6')$ are obvious. 
And we already know from Subsection \ref{ssec:cycle1} that $(1)  \Leftrightarrow (6')$.
That (1) implies (5) follows from the next proposition, using the analogue 
of Remark \ref{rem:shift}.

\begin{proposition}\label{prop:inverse}
Let $\fM$ be a weight matrix satisfying  
\eqref{fM_{rai}}.
Let $f : U \to V$ be a $C^\infty$-mapping between open subsets 
$U \subseteq E$ and $V \subseteq F$ of Banach spaces $E, F$ satisfying 
\begin{align} \label{eq:finv}
\begin{split}
  \exists \la \in \La ~\exists C,\rh \ge1 &~\forall k \in \N ~\forall  x \in U : 
  \|f^{(k)}(x)\|_{L^{k}(E,F)} \le C \rh^{k-1} (k-1)! M^\la_{k-1}
\end{split}
\end{align}
and so that $f'(x_0) \in L(E,F)$ is invertible. 
Then there exist neighborhoods $x_0 \in U_0 \subseteq U$ and $f(x_0) \in V_0 \subseteq V$ and a 
$C^\infty$-mapping $g : V_0 \to U_0$ satisfying  
\begin{align} \label{eq:ginv}
\begin{split}
  \exists \mu \in \La ~\exists D,\si \ge1 &~\forall k \in \N ~\forall  y \in V_0 : 
  \|g^{(k)}(y)\|_{L^{k}(F,E)} \le D \si^{k-1} (k-1)! M^\mu_{k-1}
\end{split}
\end{align}
and such that $f \o g = \on{id}_{V_0}$.
\end{proposition}

\begin{proof}
  We adapt the proof of \cite{Koike96}. Different arguments were given in \cite{Komatsu79} and \cite{Yamanaka89}, 
  under more restrictive assumptions also in \cite{Dynkin80} and \cite{BM04}.
  By the classical $C^\infty$ inverse mapping theorem 
  there exist neighborhoods $x_0 \in U_0 \subseteq U$ and $f(x_0) \in V_0 \subseteq V$ and a 
  $C^\infty$-mapping $g : V_0 \to U_0$ such that $f \o g = \on{id}_{V_0}$.
  We can assume that $(f'(x))^{-1}$ is bounded for $x \in U_0$.
  We shall show that $g$ satisfies \eqref{eq:ginv}.

  Let $S_k \in C^\infty(U_0, L(F,E))$, $k\ge 1$, be given and define $R_k(x)$, $k \ge 0$, for $x \in U_0$ recursively by setting
  \begin{equation} \label{eq:rec}
    R_0(x) := \id_E, \quad R_k(x) := (R_{k-1}(x) S_k(x))'.
  \end{equation}
  Thus $R_{k-1} \in C^\infty(U_0, L(E,L^{k-1}(F,E)))$ and $R_{k-1}(x) S_k(x) \in L^k(F,E)$. It follows that 
  \[
    \|R_k(x)\|_{L(E,L^{k}(F,E))} \le \sum_{\substack{\be_1+\cdots+\be_k=k\\\be_i\ge 0}} N(\be_1,\ldots,\be_k) \prod_{i=1}^k  
    \|S_i^{(\be_i)}(x)\|_{L^{\be_i}(E,L(F,E))},
  \]
  where the nonnegative integers $N(\be_1,\ldots,\be_k)$ are given by the identity
  \[
    \sum_{\substack{\be_1+\cdots+\be_k=k\\\be_i\ge 0}} N(\be_1,\ldots,\be_k) \prod_{i=1}^k  t_i^{\be_i} 
    = \prod_{j=1}^k \sum_{\ell=1}^j t_\ell.
  \]
  Since $\prod_{j=1}^k \sum_{\ell=1}^j t_\ell \le (\sum_{\ell=1}^k t_\ell)^k = \sum k! \prod_{i=1}^k  \frac{t_i^{\be_i}}{\be_i!}$, where
  the sum is taken over all $\be_i\ge 0$ so that $\be_1+\cdots+\be_k=k$, we obtain
  \begin{equation}  \label{eq:RS}
    \|R_k(x)\|_{L(E,L^{k}(F,E))} \le \sum_{\substack{\be_1+\cdots+\be_k=k\\\be_i\ge 0}} k! \prod_{i=1}^k  \frac{\|S_i^{(\be_i)}(x)\|_{L^{\be_i}(E,L(F,E))}}{\be_i!}.
  \end{equation}

  If we set $S(x) = S_k(x):= (f'(x))^{-1}$ for all $k\ge 1$, then for $n\ge 1$ 
  \[
    g^{(n)}(y) = R_{n-1}(x) S(x), \quad (x=g(y)), 
  \]
  where the sequence $R_n$ is defined by \eqref{eq:rec}. 
  Applying Lemma \ref{lem:inverseclosed} to $T = f'$ and using \eqref{eq:RS}, we find that there exist $\mu,\nu \in \La$ and $D,H,\si,\ta \ge1$ so that
  \begin{align*}
    \|R_k(x)\|_{L(E,L^{k}(F,E))} \le \sum_{\substack{\be_1+\cdots+\be_k=k\\\be_i\ge 0}} k! (D\si)^k \prod_{i=1}^k    M^\mu_{\be_i} 
    \le H \ta^k k! M^\nu_k,
  \end{align*}
  by \eqref{fM_{rai}}. This implies \eqref{eq:ginv}.
\end{proof}

\section{Proof of Theorem \ref{thm:bfM}: the Beurling case} \label{sec:proofB}

The structure of the proof of the equivalence of the six items in Theorem \ref{thm:bfM} is again 
represented by the diagram in \eqref{eq:diag}, where now: 
\begin{enumerate}
    \item[$(6')$] If $f \in \EbfM(\R)$ and $f(0) \ne 0$ then $1/f$ is $\EbfM$ on its domain of definition.  
\end{enumerate}

\subsection{The equivalences \texorpdfstring{$(1) \Leftrightarrow (2) \Leftrightarrow (3)$}{(1)<=>(2)<=>(3)}}

Lemma~\ref{lem:raiFdB} implies $(1) \Leftrightarrow (2)$, since 
\eqref{fM_H} follows from \eqref{fM_(Com)}.
The equivalence of (2) and (3)
was shown in \cite[4.11]{RainerSchindl12}.

\subsection{The cycle \texorpdfstring{$(1) \Rightarrow (4) \Rightarrow (6') \Rightarrow (1)$}{(1)=>(4)=>(6')=>(1)}} \label{ssec:cycle1B}

That (4) implies $(6')$ follows in the same way as in the Roumieu case, see Subsection \ref{ssec:cycle1}. 
The implication $(6')$ $\Rightarrow$ (1) is a consequence of the following lemma.

\begin{lemma}
  If $\EbfM(\R)$ is inverse closed, then $\fM$ satisfies \eqref{fM_(rai)}. 
\end{lemma}

\begin{proof}
   We follow an argument of \cite{Bruna80/81}. 
   Consider the algebra $\cA := \EbfM(\R)$ and its subalgebra $\cB := \{f \in \cA : \|f\|_\infty :=\|f\|_{L^\infty(\R)}< \infty\}$.
Endow $\cB$ with the topology generated by all seminorms $Q := \{\|~\|^{M^\la}_{K,\rh}\}_{\la,K,\rh} \cup\{\|~\|_{\infty}\}$. 
Then $\cB$ is a Fr\'echet algebra.
If $f,g \in \cB$ are such that $|f(x)| \ge c >0$ for all $x \in \R$ and $\|f-g\|_{\infty} \le c/2$, then 
$|g(x)| \ge c/2 >0$, i.e., the set $\{f \in \cB : 1/f \in \cB\}$ is open in $\cB$. 
By \cite[Thm~13.17]{Zelazko65}, see also \cite[Prop~5.2]{Bruna80/81}, we may conclude that the algebra $\cB$ is locally m-convex, i.e.,
$\cB$ has an equivalent seminorm
  system $P=\{p\}$ such that $p(fg) \le p(f)p(g)$ for all $f,g \in \cB$. 
  So for each $\la\in \La$, compact $K \subseteq \R$, and $\rh>0$ there exist $p \in P$, $q \in Q$, 
  and constants $C,D >0$ such that
  \[ 
  \|f^m\|^{M^{\la}}_{K,\rh} \le C p(f^m) \le C p(f)^m \le C D^m q(f)^m, \quad f \in \cB, m \in \N.
  \]
  We shall use this inequality for the functions $f_t(x) = e^{itx}$, and, 
  since $\|f_t\|_\infty = 1 \le \|f_t\|^{M^\mu}_{[-a,a],\si}$ for each $\mu \in \La$ and $a,\si>0$, 
  we can replace $q$ in the very same inequality by some seminorm $\|~\|^{M^\mu}_{[-a,a],\si}$.
Then the proof of \cite[4.11$(2)\Rightarrow(3)$]{RainerSchindl12} yields \eqref{fM_(rai)}.
\end{proof}
  
The remaining implication (1) $\Rightarrow$ (4) (as well as (1) $\Rightarrow$ (5) below) 
we shall deduce from the corresponding result in the Roumieu case by means of the following lemma, 
which is a variation of \cite[Lemma~6]{Komatsu79b}.

\begin{lemma} \label{Komatsu}
  Let $L\in \R_{\ge 0}^\N$ and $M^1,M^2,M^3 \in \R_{>0}^\N$ 
  be sequences satisfying $L \lhd M^1 \le M^2 \le M^3$ and $(M^i_k)^{1/k} \to \infty$ for $i=1,2,3$, and assume that 
  there exist $1 \le H_1 \le H_2$ so that 
  \begin{equation} \label{eq:cond123}
    (M^1_j)^{1/j} \le H_1 (M^2_k)^{1/k} \le H_2 (M^3_\ell)^{1/\ell}   \quad \text{ for } j \le k \le \ell.
  \end{equation} 
  Then there exist sequences $N^1, N^2 \in \R_{>0}^\N$ with $L \le N^1 \le N^2 \lhd  M^3$  satisfying  
  $(N^i_k)^{1/k} \to \infty$ for $i=1,2$ and so that 
  \begin{equation} \label{eq:cond12}
    (N^1_j)^{1/j} \le \sqrt{H_1} (N^2_k)^{1/k}    \quad \text{ for } j \le k.
  \end{equation}
\end{lemma}

\begin{proof}
  Without loss of generality we may assume that $L_k>0$ for all $k$; 
  otherwise replace $L$ by $\bar L$ where $\bar L_k = L_k$ if $L_k>0$ and $\bar L_k = 1$ if $L_k=0$ 
  (we still have $\bar L \lhd M^1$ since $(M^1_k)^{1/k} \to \infty$). 
  The sequences, for $i=1,2,3$,  
  \[
    c^i_k := \Big(\frac{M^i_k}{L_k}\Big)^{\frac 1 k}
  \]
  satisfy $c^i_k \to \infty$, since $L \lhd M^i$.
  We define, for $i=1,2$, 
  \begin{equation} \label{eq:Ndef}
    (N^i_k)^{\frac{1}{k}} := 
  \max\Big\{\sqrt{(M^i_k)^{\frac{1}{k}}},\max_{j \le k}
   \frac{(M^i_j)^{\frac{1}{j}}}{c^i_j} 
  \Big\}
  = 
  \max\Big\{\sqrt{(M^i_k)^{\frac{1}{k}}},\max_{j \le k}
    L_j^{\frac 1 j} 
  \Big\}.  
  \end{equation}
  Then clearly $(N^i_k)^{1/k} \to \infty$ and $L \le N^1 \le N^2$.
  For each $\ep>0$ there exists $j_{\ep,i}$ so that $1/c^i_j \le \ep$ for $j>j_{\ep,i}$.
  Thus, by \eqref{eq:cond123},
  \begin{align*}
  \Big(\frac{N^i_k}{M^{i+1}_k}\Big)^{\frac{1}{k}} 
  &\le
  \max\Big\{(M^{i+1}_k)^{-\frac{1}{2k}}, (M^{i+1}_k)^{-\frac{1}{k}}\max_{j \le k}
   \frac{(M^i_j)^{\frac{1}{j}}}{c^i_j  }
  \Big\} 
  \\
  & \le 
  \max\Big\{(M^{i+1}_k)^{-\frac{1}{2k}}, 
  (M^{i+1}_k)^{-\frac{1}{k}} \max_{j \le j_{\ep,i}}
   \frac{(M^{i}_j)^{\frac{1}{j}}}{c^i_j }, H_i \ep
  \Big\}
  \le H_i \ep, 
  \end{align*}
  for $k$ sufficiently large, i.e., $N^1 \lhd M^2$ and $N^2 \lhd M^3$.
  It remains to show \eqref{eq:cond12}.
  But this is immediate from \eqref{eq:cond123} and \eqref{eq:Ndef}, indeed for $j \le k$,
  \begin{align*}
     (N^1_j)^{\frac{1}{j}} =  
      \max\Big\{\sqrt{(M^1_j)^{\frac{1}{j}}},\max_{h \le j}
      L_h^{\frac 1 h}
      \Big\} 
      \le  
      \max\Big\{\sqrt{H_1(M^2_k)^{\frac{1}{k}}},\max_{h \le k}
      L_h^{\frac 1 h} \Big\}
      \le \sqrt{H_1} (N^2_k)^{\frac{1}{k}}
  \end{align*}
  as required. 
\end{proof}

The following proposition implies (1) $\Rightarrow$ (4), by the analogue of Remark \ref{rem:shift}.

\begin{proposition}\label{prop:ODEb}
Let $\fM$ be a weight matrix satisfying
\eqref{fM_(rai)} and \eqref{fM_(Com)}.
Let $X$ be a Banach space and let $f : W \to X$ be an $C^\infty$-mapping defined in an open 
subset $W \subseteq X \times \R$ and satisfying
\begin{align} \label{eq:fb}
\begin{split}
  \forall \nu \in \La, \rh>0  ~\exists C \ge1 &~\forall (k,\ell) \in \N^2 ~\forall  (x,t) \in W : 
  \\
  &\|f^{(k,\ell)}(x,t)\|_{L^{k,\ell}(X,\R;X)} \le C \rh^{k+\ell} (k+\ell-1)! M^\nu_{k+\ell-1}.
\end{split}
\end{align}
Then the solution $x : I \to X$ of the initial value problem
\begin{equation} \label{eq:IVPb}
  x'(t) = f(x(t),t), \quad x(0) = x_0,  
\end{equation}
satisfies
\begin{align} \label{eq:xb}
  \forall \la \in \La, \si>0 ~\exists D \ge1 ~\forall k \in \N ~\forall  t \in I : 
  \|x^{(k)}(t)\| \le D \si^{k} (k-1)! M^\la_{k-1}.
\end{align}
\end{proposition}

\begin{remark}
  The analogue of Remark \ref{rem:1} applies.
\end{remark}

\begin{proof}
  In the same way as in the Roumieu case (Proposition \ref{prop:ODEr}) we may reduce to the initial value problem \eqref{eq:IVP2}, 
  where now $g \in C^\infty(U,Y)$ satisfies 
  \begin{align} \label{eq:g2Beu}
     \forall \nu \in \La, \rh>0 ~\exists C \ge1 ~\forall k \in \N ~\forall  y \in U : \|g^{(k)}(y)\|_{L^k(Y,Y)} \le C \rh^k (k-1)! M^\nu_{k-1}.
  \end{align} 
  There is a unique solution $y \in C^\infty(I,Y)$ defined on some interval $I$. 
  Let $\la \in \La$ be fixed. 
  By \eqref{fM_(rai)} 
  there exist $\mu, \nu \in \La$ and $H,J\ge 1$ such that
  \begin{equation} \label{eq:lamunu}
    (M^\nu_j)^{1/j} \le H (M^\mu_k)^{1/k} \le J (M^\la_\ell)^{1/\ell}   \quad \text{ for } j \le k \le \ell,
  \end{equation} 
  and, by \eqref{fM_(Com)}, 
  \begin{equation} \label{eq:unbounded}
    \lim (M^\nu_k)^{1/k} =\lim (M^\mu_k)^{1/k} =\lim (M^\la_k)^{1/k} =\infty.  
  \end{equation}
  We may assume without loss of generality that $\nu \le \mu \le \la$ and thus $M^\nu \le M^\mu \le M^\la$.
  If we set 
  \begin{align*}
    L_{k-1} := \sup_{y \in U}  \tfrac{1}{(k-1)!} \|g^{(k)}(y)\|_{L^k(Y,Y)},
  \end{align*}
  then \eqref{eq:g2Beu} implies $L \lhd M^\nu$.
  Thus, by applying Lemma \ref{Komatsu}, 
  we find sequences $N^1$ and $N^2$ with $L \le N^1 \le N^2 \lhd  M^\la$  satisfying  
  $(N^i_k)^{1/k} \to \infty$ for $i=1,2$ and so that 
  \begin{equation*} 
    (N^1_j)^{1/j} \le \sqrt{H} (N^2_k)^{1/k}    \quad \text{ for } j \le k.
  \end{equation*}
  Repeating the proof of Proposition \ref{prop:ODEr} (with $N^1$ in place of $M^\la$ and $N^2$ in place of $M^\mu$)
  we may conclude that
  \begin{align*} 
  \exists D,\si \ge1 ~\forall k \in \N ~\forall  t \in I : 
  \|y^{(k)}(t)\| \le D \si^{k} (k-1)! N^2_{k-1}
  \end{align*}
  which implies
  \begin{align*} 
  \forall \ta>0 ~\exists E \ge1 ~\forall k \in \N ~\forall  t \in I : 
  \|y^{(k)}(t)\| \le E \ta^{k} (k-1)! M^\la_{k-1},
  \end{align*}
  since $N^2 \lhd  M^\la$. 
  As $\la$ was arbitrary, the result follows.
\end{proof}

\subsection{The cycle \texorpdfstring{$(1) \Rightarrow (5) \Rightarrow (6) \Rightarrow (6') \Rightarrow (1)$}{(1)=>(5)=>(6)=>(6')=>(1)}} 

Obviously, $(5) \Rightarrow (6) \Rightarrow (6')$, and  
$(1) \Leftrightarrow (6')$, by Subsection \ref{ssec:cycle1B}.
Finally, (1) $\Rightarrow$ (5) follows from the following proposition.

\begin{proposition}\label{prop:inverseB}
Let $\fM$ be a weight matrix satisfying \eqref{fM_(rai)} and \eqref{fM_(Com)}.
Let $f : U \to V$ be a $C^\infty$-mapping between open subsets 
$U \subseteq E$ and $V \subseteq F$ of Banach spaces $E, F$ satisfying 
\begin{align} \label{eq:finvB}
\begin{split}
  \forall \nu \in \La, \rh>0 ~\exists C \ge1 &~\forall k \in \N ~\forall  x \in U : 
  \|f^{(k)}(x)\|_{L^{k}(E,F)} \le C \rh^{k-1} (k-1)! M^\nu_{k-1}
\end{split}
\end{align}
and so that $f'(x_0) \in L(E,F)$ is invertible. 
Then there exist neighborhoods $x_0 \in U_0 \subseteq U$ and $f(x_0) \in V_0 \subseteq V$ and a 
$C^\infty$-mapping $g : V_0 \to U_0$ satisfying 
\begin{align} \label{eq:ginvB}
\begin{split}
  \forall \mu \in \La,\si>0 ~\exists D \ge1 &~\forall k \in \N ~\forall  y \in V_0 : 
  \|g^{(k)}(y)\|_{L^{k}(F,E)} \le D \si^{k-1} (k-1)! M^\mu_{k-1}
\end{split}
\end{align}
and such that $f \o g = \on{id}_{V_0}$.
\end{proposition}

\begin{proof}
  Let $\la \in \La$ be fixed. 
  By \eqref{fM_(rai)} and \eqref{fM_(Com)}, 
  there exist $\mu, \nu \in \La$ satisfying \eqref{eq:lamunu} and \eqref{eq:unbounded}. 
  If we set 
  \begin{align*}
    L_{k-1} := \sup_{x \in U}  \tfrac{1}{(k-1)!} \|f^{(k)}(x)\|_{L^k(E,F)},
  \end{align*}
  then \eqref{eq:finvB} implies $L \lhd M^\nu$.
  In analogy to the proof of Proposition \ref{prop:ODEb}, 
  we repeat the proof of Proposition \ref{prop:inverse} (and that of Lemma \ref{lem:inverseclosed}) with the sequences 
  $N^1$ and $N^2$ provided by Lemma \ref{Komatsu}.  
\end{proof}

\section{Proof of Theorems \ref{thm:rM}, \ref{thm:bM}, \ref{thm:rom}, and \ref{thm:bom}} \label{sec:proofMom}

\subsection{Proof of Theorems \ref{thm:rM} and \ref{thm:bM}}

Apply Theorems \ref{thm:rfM} and \ref{thm:bfM} to the constant weight matrix $\fM = \{M\}$.

\subsection{Proof of Theorems \ref{thm:rom} and \ref{thm:bom}}

For a weight function $\om$ and each $\rh>0$ consider the sequence $\Om^\rh \in \R_{>0}^\N$ defined by 
\[
\Om^\rh_k:=\tfrac{1}{k!}\exp(\tfrac{1}{\rh} \vh^*(\rh k)).
\]
By the properties of $\vh^*$,
the collection $\fW := \{\Om^\rh : \rh>0\}$ forms a weight matrix, and we have 
$\cE^{[\om]}(U) = \cE^{[\fW]}(U)$ 
as locally convex spaces, by \cite{RainerSchindl12}.
Moreover, $\fW$ 
satisfies \eqref{fM_(dc)} as well as \eqref{fM_{dc}}. 
If $\om(t)=O(t)$ as $t\to \infty$, then $\fW$ satisfies \eqref{fM_H}, 
and if $\om(t)=o(t)$ as $t\to \infty$, then $\fW$ satisfies \eqref{fM_(Com)}.
Thus Theorems \ref{thm:rom} and \ref{thm:bom} are immediate consequences of Theorems~\ref{thm:rfM}, \ref{thm:bfM},   
  and \cite[6.3, 6.5]{RainerSchindl12}.

\appendix

\section{Weight sequences as required in Remark \ref{rem:rai}} \label{appendix}

Let us now find explicit sequences that satisfy the requirements of Remark~\ref{rem:rai}. To this end we 
construct a weight sequence $M=(M_k)$ such that $(M_{k+1}/M_k)^{1/k}$ is bounded, $M^{1/k}_k$ tends to $\infty$ but 
is not almost increasing, and $k!^s \le M_k \le k!^t$ for suitable $s,t > 0$ and sufficiently large $k$. 
Since for every Gevrey sequence $G^s =(k!^s)_k$ (where $s\ge 0$), $(G^s_k)^{1/k}$ is increasing (and tends to $\infty$ if $s>0$), 
the pair of sequences $(G^s, M)$, or $(M,G^t)$, will fulfill the requirements of Remark~\ref{rem:rai} (after adjusting finitely many 
terms of one sequence).

Let $k_j:=2 \uparrow \uparrow j = 2^{2^{{\cdot}^{{\cdot}^ 2}}}$ ($j$ times) for $j\ge 1$ and $k_0:=0$. 
Let $\vh : [0,\infty) \to [0,\infty)$ be the function whose graph is the polygon with vertices $\{v_j=(k_j,\vh(k_j)): j \in \N\}$
defined by
\begin{align*}
  \vh(0):=0, \quad \vh(2):= 8 \log 2,\quad
  \vh(k_j) := 
  \begin{cases}
    k_j \log k_{j+1} & \text{ $j$ even} \\
    k_j \log (k_j k_{j-2}) & \text{ $j$ odd}
  \end{cases}, \quad (j\ge 2).
\end{align*} 
We claim that the sequence $M=(M_k)$ defined by $M_k:= \exp(\vh(k))/k!$ satisfies:
\begin{enumerate}
   \item $M$ is a weight sequence, i.e., $M$ is weakly log-convex, 
   \item $\sup_k (\frac{M_{k+1}}{M_k})^{1/k}< \infty$,
   \item $M^{1/k}_k \to \infty$,
   \item $M_k^{1/k}$ is not almost increasing,
   \item $k!^s \le M_k \le k!^t$ for all $0\le s \le 1/4$, $t> 3$, and all $k\ge k_0(t)$.
\end{enumerate}

To see that $M$ is weakly log-convex it suffices to show that the slopes of the line segments in the graph of $\vh$ are 
increasing. Let $a_j$ denote the slope of the line segment left of the vertex $v_j$.
Then, for $i\ge 2$,
\begin{align} \label{eq:A1}
\begin{split}
  a_{2i-1} &= \frac{k_{2i-1} \log (k_{2i-1} k_{2i-3}) - k_{2i-2} \log k_{2i-1}}{k_{2i-1}-k_{2i-2}} = \frac{\frac{5}{4} k_{2i-2}-1}{k_{2i-2}-1} \log k_{2i-1} \\
  a_{2i} &= \frac{k_{2i} \log k_{2i+1} - k_{2i-1} \log (k_{2i-1} k_{2i-3})}{k_2i-k_{2i-1}} = \frac{4 k_{2i-1}-\frac{5}{4}}{k_{2i-1}-1} \log k_{2i-1} \\
  a_{2i+1} &= \frac{k_{2i+1} \log (k_{2i+1} k_{2i-1}) - k_{2i} \log k_{2i+1}}{k_{2i+1}-k_{2i}} = \frac{5 k_{2i}-4}{k_{2i}-1} \log k_{2i-1}
\end{split}
\end{align}
and $a_{2i-1} \le a_{2i} \le a_{2i+1}$.
This proves (1). 
Let us check (2). Since
\begin{align*}
  \Big(\frac{M_{k+1}}{M_k}\Big)^{1/k} = \frac{\exp(\frac{\vh(k+1)}{k}-\frac{\vh(k)}{k})}{(k+1)^{1/k}} \le \exp(\tfrac{\vh(k+1)}{k}-\tfrac{\vh(k)}{k}), 
\end{align*}
it suffices to show that $\tfrac{\vh(k+1)}{k}-\tfrac{\vh(k)}{k}$ is bounded or equivalently that the slope of the line segments of $\vh$
increases at most linearly in $k$. This is obvious from \eqref{eq:A1}.
Thanks to $k! \le k^k \le e^k k!$ we have 
\[
  \frac{\exp(\frac{\vh(k)}{k})}{k} \le M^{1/k}_k \le \frac{\exp(\frac{\vh(k)}{k}+1)}{k}.
\]
This implies (3). To show (4) let $j$ be even. Then 
\begin{align*}
  \log \frac{M^{1/k_j}_{k_j}}{M^{1/k_{j+1}}_{k_{j+1}}} \ge \log k_j + \log k_{j+1} - \log (k_{j+1} k_{j-1}) -1 = \log k_{j-1}-1 \to \infty,
\end{align*}
as required.
Finally, $k!^s \le M_k \le k!^t$ is equivalent to $1+s \le \tfrac{\vh(k)}{\log k!} \le 1+t$. 
We have   
\[
  \frac{\vh(k_j)}{k_j\log k_j} = 
  \begin{cases}
    2 & \text{ $j$ even} \\
    1+\frac{1}{4} & \text{ $j$ odd}
  \end{cases}
\]
and $\tfrac{\vh(k)}{k\log k} \le \tfrac{\vh(k)}{\log k!} \le 2 \tfrac{\vh(k)}{k\log k}$. 
So the first inequality in (5) follows  
thanks to the fact that $k!^s$ is log-convex for each $s\ge 0$. For the second inequality we observe that in view of \eqref{eq:A1} the slope $a_i$ 
is dominated by the increment $\log ((k+1)!^{1+t}) - \log (k!^{1+t}) = (1+t) \log (k+1)$ for all $k_{i-1}\le k < k_i$ provided that $t>3$ and that $i$
is sufficiently large depending on $t$.

\def\cprime{$'$}
\providecommand{\bysame}{\leavevmode\hbox to3em{\hrulefill}\thinspace}
\providecommand{\MR}{\relax\ifhmode\unskip\space\fi MR }
\providecommand{\MRhref}[2]{%
  \href{http://www.ams.org/mathscinet-getitem?mr=#1}{#2}
}
\providecommand{\href}[2]{#2}

\end{document}